\documentclass[12pt]{article}
\usepackage{amsmath,amsthm,amssymb,amsfonts,mathrsfs,multicol}
\newtheorem{theorem}{Theorem}[section]
\newtheorem{lemma}[theorem]{Lemma}
\newtheorem{proposition}[theorem]{Proposition}

\newtheorem{definition}[theorem]{Definition}

=22cm\headheight=0cm\topskip=0cm\topmargin=0cm

\def\Rn{\mathbb R}
\def\Sc{\mathcal S}
\def\v{\mathbf{v}}
\def\0{\sf 0}

\makeatletter
\def\dotminussym#1#2{%
  \setbox0=\hbox{$\m@th#1-$}%
  \kern.5\wd0%
  \hbox to 0pt{\hss\hbox{$\m@th#1-$}\hss}%
  \raise.8\ht0\hbox to 0pt{\hss$\m@th#1.$\hss}%
  \kern.5\wd0}

\newcommand{\bunderline}[1]{\underline{#1\mkern-4mu}\mkern4mu }

\begin{document}

\begin{center}
{\Large\sc Affine modal propositional logic} 
\bigskip

{\bf Hafez K. D.\\ Seyed-Mohammad Bagheri\footnote{Corresponding author}}

\vspace{3mm}

{\footnotesize Department of Pure Mathematics, Faculty of Mathematical Sciences,\\
Tarbiat Modares University, Tehran, Iran, P.O. Box 14115-134\\
hafez.kamrani@modares.ac.ir,\ \  bagheri@modares.ac.ir}
\end{center}

\begin{abstract}
Topological semantics for affine modal propositional logic is introduced.
The interior operator on subsets is replaced with the lower semi-continuous envelope operator
on functions. Completeness and affine compactness theorems are proved for this logic.
\end{abstract}


\bigskip

The topological semantics for propositional modal logic has a rich history,
actually predating the relational semantics primarily attributed to Kripke (see
\cite{Goldblatt}). Mckinsey and Tarski first established its completeness with respect
to S4 and demonstrated its decidability in their classical paper \cite{Tarski}.
In modern terms, a topological model is a tuple $M=(X,\v)$, where $X$ is a
topological space and $\v$ a valuation assigning a subset of $X$ to every
formula. The interpretations are straightforward: for a formula $\phi$ and
$x\in X$ we have $X\vDash_x\phi$ if and only if $x\in\v(\phi)$. The modal
operator $\Box$ is then interpreted simply as interior operator, i.e. $\v(\Box\phi)=\operatorname{Int}(\v(\phi))$
(See \cite{vanBenthem}).

While a Kripke semantics has been established for propositional and first-order
modal continuous logics in \cite{baratella_1,baratella_2}, this paper aims to develop topological
semantics for the propositional case, focusing specifically on its affine fragment. Recall that continuous logic \cite{ben}
is a generalization of classical logic, replacing the binary value space $\{0,1\}$ with the real line $\Rn$.
It enjoys the familiar system of connectives $\{+,r\cdot,\wedge,\vee,1\}_{r\in\Rn}$.
Affine continuous logic \cite{bagheri} is a weakening of continuous logic obtained by removing the non-linear
connectives (i.e., $\wedge,\vee$). Usual tools and technics in full continuous logic find affine counterparts in this fragment.

In topological semantics for continuous logic subsets are replaced with functions.
To each atomic proposition $p$ is assigned a function $\v_p(x):X\rightarrow[0,1]$.
The operator $\Box$ is then interpreted as the lower semi-continuous envelope of functions.
This operation coincides with the interior operation if we identify every subsets of $X$ with the corresponding characteristic function.
In this way, the functional semantics generalizes the classical one. We prove an appropriate completeness theorem for this semantics.
Section \ref{intro} reviews the preliminary definitions and properties of
semi-continuous functions. Section \ref{sem} introduces the functional topological semantics.
Section \ref{proof} presents the proof system and establishes the necessary tools. Section
\ref{complete} is dedicated to showing the completeness of the proof system with respect to topological semantics
using the canonical model. In Section \ref{compact} affine compactness theorem is proved
using ultramean construction. We also observe that canonical model used in the completeness proof is compact convex.

\section{Introduction}\label{intro}
All real valued functions on topological spaces are assumed to be bounded.
Let $X$ be a topological space. A function $f:X\rightarrow\Rn$ is \emph{lower semi-continuous} (l.s.c.) if
for each $r$, $f^{-1}(r,\infty)$ is open. It is \emph{upper semi-continuous} (u.s.c.) if for each $r$,
$f^{-1}(-\infty,r)$ is open.
The pointwise supremum of an arbitrary family of lower semi-continuous functions (if defined) is lower semi-continuous.
Similarly, the pointwise infimum of any family of upper semi-continuous functions is upper semi-continuous \cite{Aliprantis-Inf}.
The lower and upper semi-continuous envelopes of $f:X\rightarrow\Rn$ are respectively defined by
$$\bunderline{f}(x)=\sup\{h(x): \ f\geqslant h\ \mbox{is\ l.s.c.}\}.$$
$$\bar{f}(x)=\inf\{h(x): \ f\leqslant h\ \mbox{is\ u.s.c.}\}.$$
Then, $\bunderline{f}$ is lower semi-continuous and $\bar{f}$ is upper semi-continuous.
It can be also proved that $$\bunderline{f}(x)=\liminf_{\ t\rightarrow x}f(t)=\sup\big\{\inf_{t\in U}f(t):\ U\ni x\ \mbox{is\ open}\big\}$$
$$\bar{f}(x)=\limsup_{\ t\rightarrow x}f(t)=\inf\big\{\sup_{t\in U}f(t):\ U\ni x\ \mbox{is\ open}\big\}.$$
One verifies that $\bar{f}=-(\bunderline{-f})$ and that if $f,g$ are lower (resp. upper) semi-continuous, then so are
$f+g$, $f\cdot g$, $f\vee g$ and $f\wedge g$.
The following equalities and inequalities are verified easily:

\begin{multicols}{3}
	\begin{itemize}
		\item $\bunderline{f}+\bunderline{g}\leqslant\bunderline{f+g}$
		
		\item $\overline{f+g}\leqslant\bar{f}+\bar{g}$
		\vspace{1mm}
		
		\item $\bunderline{f}\vee\bunderline{g}\leqslant\underline{f\vee g}$
		\vspace{1mm}
		
		\item $\overline{f\vee g}=\bar{f}\vee\bar{g}$.
		\vspace{1mm}
		
		\item $\overline{f\wedge g}\leqslant\bar{f}\wedge\bar{g}$
		\vspace{1mm}
		
		\item $\underline{f\wedge g}=\bunderline{f}\wedge\bunderline{g}$.
	\end{itemize}
\end{multicols}

The indicator function of $A\subseteq X$ is defined by $1_{A}(x)=1$ if $x\in A$ and $1_A(x)=0$ otherwise.
Then, $1_A$ is lower (resp. upper) semi-continuous if and only if $A$ is open (resp. closed).
Generally, the lower semi-continuous envelope of $1_A$ is $1_{A^{\circ}}$ and upper semi-continuous envelope
of $1_A$ is $1_{\overline A}$ where $A^{\circ}$ and $\overline A$ are the interior and closure of $A$ respectively.

\section{Semi-continuous envelope semantics}\label{sem}

Let $\mathbb P=\{p,q,...\}$ be a countable family of atomic propositions.
Compound propositions for affine modal propositional logic AML (affine propositions) are inductively defined by
$$1,\ \ \ p,\ \ \ \phi+\psi,\ \ \ r\phi,\ \ \ \Box\phi$$
where, $r\in\Rn$. Similarly, compound `non-affine' propositions are inductively defined by
$$1,\ \ \ p,\ \ \ \phi+\psi,\ \ \ r\phi,\ \ \ \Box\phi,\ \ \ \phi\wedge\psi,\ \ \ \phi\vee\psi.$$

To each proposition is assigned a bound inductively defined as follows:
\begin{multicols}{2}
\begin{itemize}
\item $\sf{b}_1=1$
\item $\sf{b}_p=1$
\item $\sf{b}_{r\phi}=|r|\sf{b}_\phi$
\item $\sf{b}_{\phi+\psi}=\sf{b}_\phi+\sf{b}_\psi$
\item $\sf{b}_{\Box\phi}=\sf{b}_\phi$
\item $\sf{b}_{\phi\wedge\psi}=\sf{b}_{\phi\vee\psi}=\max\{\sf{b}_\phi,\sf{b}_\psi\}$
\end{itemize}
\end{multicols}

\begin{definition} \emph{A \emph{topological model} (or model for short) is a pair $M=(X,\v)$ where $X$ is a
topological space and $\v$ is a valuation function assigning to every atomic proposition $p$ a map $\v_p:X\rightarrow[0,1]$.}
\end{definition}
\vspace{1mm}

For a proposition $\phi$ and model $M=(X,\v)$, \ $\phi^M(x)$ is the value of $\phi$ in $x\in X$ which is defined
by induction on the complexity of $\phi$ as follows:
\vspace{1mm}

$\bullet$\ \ for atomic proposition $p$,\ \ $p^M(x)=\v_p(x)$. Also, $1^M(x)=1$

$\bullet$\ \ $+,\ r\cdot,\ \wedge,\ \vee$ are defined in the obvious way, e.g. $(\phi+\psi)^M(x)=\phi^M(x)+\psi^M(x)$

$\bullet$\ \ $(\Box\phi)^M(x)=\bunderline{\phi^M}(x)$.
\bigskip

As an example, let $X=[0,1]^{\mathbb P}$ be endowed with the Tychonoff topology.
Then, in the evaluation model where $\v_p(x)=x(p)$, one has that $\Box\phi=\phi$ for every $\phi$.
In contrast, the integer part valuation $\v_p(x)=\lfloor x(p)\rfloor$ gives a nontrivial topological model.

We use $\lozenge\phi$ as an abbreviation for $-\Box-\phi$. It is then proved that
$$(\lozenge\phi)^M(x)=\overline{\phi^M}(x).$$
One also verifies that for every $M$ and model $\phi$,\ \ $-{\sf b}_\phi\leqslant\phi^M\leqslant{\sf b}_\phi$.
In the classical topological semantics \cite{Tarski}, $\phi^M$ is a subset of $X$ and $(\Box\phi)^M$
(resp. $(\lozenge\phi)^M$) is interpreted as the interior (resp. closure) of the set $\phi^M$.
This is a special case of the present definition if we identify $A$ with $1_A$.

An expression of the form $\phi\leqslant\psi$ is called a \emph{condition}.
Let $M=(X,\v)$ be a model and $x\in X$. Then, by $M\vDash_x\phi\leqslant\psi$ is meant $\phi^M(x)\leqslant\psi^M(x)$.
For a set $\Gamma$ of conditions, $M\vDash_x\Gamma$ has the obvious meaning.
$\Gamma$ is \emph{satisfiable} if there is a model $M=(X,\v)$ and $x\in X$ such that $M\vDash_x\Gamma$.
Finally, $\Gamma\vDash\phi\leqslant\psi$ means that for every model $M=(X,\v)$ and $x\in X$,
if $M\vDash_x\Gamma$ then $M\vDash_x\phi\leqslant\psi$.
Global variants of these notions are defined similarly.
So, for example, $M\vDash\phi\leqslant\psi$ means $M\vDash_x\phi\leqslant\psi$ for all $x\in X$
and $\Gamma$ is globally satisfiable if there is $M=(X,\v)$ such that $M\vDash_x\Gamma$ for all $x\in X$.
This is denoted by $M\vDash\Gamma$.

Although we have defined topological semantics generally, in this paper, we focus on the affine case.
So, from now on, by proposition, we always mean an affine proposition. We prove completeness and compactness
for AML. The general real valued logic has a parallel discussion.

The axiom $K$ in the classical form of modal logic is written as
$$\Box(\phi\rightarrow\psi)\rightarrow(\Box\phi\rightarrow\Box\psi).$$
The affine counterpart of axiom K is then either of the following equivalent conditions:
$$\Box(\phi-\psi)\leqslant\Box\phi-\Box\psi,\hspace{15mm} \Box\phi+\Box\psi\leqslant\Box(\phi+\psi).$$
Also, the axioms $T$ and 4 are respectively written as $\Box\phi\rightarrow\phi$ and $\Box\phi\rightarrow\Box\Box\phi$.
These axioms now take the forms $\Box\phi\leqslant\phi$ and $\Box\phi\leqslant\Box\Box\phi$ respectively.
It is clear that these conditions as well as the following ones hold in every topological model:
\vspace{1mm}

- $\Box(\phi+r)=\Box\phi+r$ 

- $\Box(r\phi)=r\Box\phi$\ \ \ \ \ \ \ \ \ \ \ \ for $0\leqslant r$ 

- $\Box(r\phi)=r\lozenge\phi$\ \ \ \ \ \ \ \ \ \ \ \ for $r\leqslant0$ 
\bigskip

Also, the classical modus ponens and necessitation rules are respectively as follows:
$$\frac{\phi,\ \phi\rightarrow\psi}{\psi},\hspace{25mm} \frac{\phi}{\Box\phi}.$$
These rules take the following forms in AML respectively:
$$\frac{\phi\leqslant\psi,\ \psi\leqslant\theta}{\phi\leqslant\theta},
\hspace{25mm} \frac{\phi\leqslant\psi}{\Box\phi\leqslant\Box\psi}.$$
Following the classical case, we define the affine variant of S4 as K+T+4 above.
So, by the necessitation rule, affine S4 implies that $\Box\Box\phi=\Box\phi$
whose interpretation is that the lower envelope of a lower semi-continuous function is the function itself.
Similarly, affine S5 can be defined as affine S4 plus the axiom $\phi\leqslant\Box\lozenge\phi$
(which is the affine counterpart of the axiom B stated by $\phi\rightarrow\Box\lozenge\phi$).

\section{Proof system}\label{proof}

By $\phi=\psi$ is meant $\{\phi\leqslant\psi, \psi\leqslant\phi\}$.
The affine variant of S4 (denoted by AS4) is defined by the following axioms and rules.
\bigskip

\noindent{\bf Logical axioms}:
\begin{multicols}{2}
	\begin{enumerate}
		\item[(A1)] $r\leqslant s$ \hspace{10mm} (if\ $\Rn\vDash r\leqslant s$)
		
		\item[(A2)] $\phi+(\psi+\theta)=(\phi+\psi)+\theta$
		
		\item[(A3)] $\phi+\psi=\psi+\phi$
		
		\item[(A4)] $0+\phi=\phi$
		
		\item[(A5)] $r(\phi+\psi)=r\phi+r\psi$
		
		\item[(A6)] $(r+s)\phi=r\phi+s\phi$
		
		\item[(A7)] $r(s\phi)=(rs)\phi$
		
		\item[(A8)] $1\phi=\phi$
		
		\item[(A9)] $0\phi=0$
		
		\item[(A10)] $0\leqslant p\leqslant1$ \hspace{4mm} $(p\in\mathbb{P}$)
		
		\item[(A11)] $\Box(\phi-\psi)\leqslant\Box\phi-\Box\psi$  \hspace{6mm} K
		
		\item[(A12)] $\Box\phi\leqslant\phi$\hspace{33mm} T
		
		\item[(A13)] $\Box\phi\leqslant\Box\Box\phi$\hspace{29mm} 4
		
		\item[(A14)] $\Box r\phi=r\Box\phi$ \hspace{23mm} ($r\geqslant0$)
		
		\item[(A15)] $\Box r=r$
	\end{enumerate}
\end{multicols}

\noindent{\bf Inference rules}:
\begin{multicols}{2}
	\begin{enumerate}
		\item[(R1)] $\dfrac{\phi\leqslant\psi}{\phi+\theta\leqslant\psi+\theta}$
		\item[(R2)] $\dfrac{0\leqslant r,\ \phi\leqslant\psi}{r\phi\leqslant r\psi}$
		\item[(R3)] $\dfrac{\phi\leqslant\psi,\ \psi\leqslant\theta}{\phi\leqslant\theta}$ \hspace{10mm} MP
		\item[(R4)] $\dfrac{\phi\leqslant\psi}{\Box\phi\leqslant\Box\psi}$\hspace{20mm} N
	\end{enumerate}
\end{multicols}
\bigskip

Below, $\Sc,\Sc_1,...$ denote conditions. If $\Sc$ is $\phi\leqslant\psi$, by $\Sc^\epsilon$ is meant $\phi\leqslant\psi+\epsilon$.
We often write $\Gamma,\Sc$ for $\Gamma\cup\{\Sc\}$.

\begin{definition} \label{proofdfn}
{\em We write $\Vdash_0\Sc$ if $\Sc$ is a logical axiom.
Suppose that $\Vdash_k$ has been defined for every $k<n$.
Then, we write $\Vdash_n\Sc$ if one of the following requirements is satisfied:
\begin{quote}

$\bullet$ $\Vdash_k\Sc$ for some $k<n$;

$\bullet$ there are $k<n$ and $\Sc_1,\Sc_2$ such that $\Vdash_k\Sc_1$,\ \ $\Vdash_k\Sc_2$
and $\frac{\Sc_1, \Sc_2}{\Sc}$ is an instance of the logical rules (R1)-(R4).

\end{quote}}
\end{definition}

We write $\Vdash\Sc$ if $\Vdash_k\Sc$ for some $k$.

\begin{definition} \label{proofdfn}
{\em We write $\Gamma\vdash_0\Sc$ if either $\Vdash\Sc$ or $\Sc\in\Gamma$.
Suppose that $\vdash_k$ has been defined for every $k<n$.
Then, we write $\Gamma\vdash_n\Sc$ if one of the following holds:
\begin{quote}

$\bullet$ $\Gamma\vdash_k\Sc$ for some $k<n$;

$\bullet$ there are $k<n$ and $\Sc_1,\Sc_2$ such that $\Gamma\vdash_k\Sc_1,\Sc_2$
and $\frac{\Sc_1, \Sc_2}{\Sc}$ is an instance of the logical rules (R1)-(R3).


\end{quote}}
\end{definition}

\begin{proposition} \label{soundness} \emph{(Soundness)}
If\ \ $\Vdash\Sc$, then $\vDash\Sc$.
\end{proposition}

\begin{definition}
{\em $\Gamma$ \emph{proves} $\Sc$ (denoted by $\Gamma\vdash\Sc$) if $\Gamma\vdash_n\Sc$ for some $n$.\ \
$\Gamma$ is \emph{inconsistent} if $\Gamma\vdash 1\leqslant 0$. Otherwise, it is \emph{consistent}.}
\end{definition}

One verifies (by induction) that \ $\vdash\Sc$ if and only if \ $\Vdash\Sc$. Proof of the following lemma is routine.

\begin{lemma} \label{easy consequence}
For each $\Gamma$, $\phi,\psi$ and $r,s$ the following hold:

{\em(i)}\  $\vdash -\sf{b}_\phi\leqslant\phi\leqslant\sf{b}_\phi$.

{\em(ii)} $r=0\vdash r\phi=0$.

{\em(iii)} $\phi=0\vdash r\phi=0$.

{\em(iv)} $\{r=s,\ \phi=\psi\}\vdash r\phi=s\psi$.

{\em(v)} $\vdash\Box\Box\phi=\Box\phi\leqslant\phi\leqslant\lozenge\phi$.    

{\em(vi)} $\vdash\Box(\Box\phi_1+\cdots+\Box\phi_k)=\Box\phi_1+\cdots+\Box\phi_k$.

{\em(vii)} $\vdash\Box(\phi+r)=\Box\phi+r$. 

{\em(viii)} If $\Gamma\vdash 1\leqslant 0$ then  $\Gamma\vdash\Sc$ for every condition $\Sc$.

\end{lemma}

\begin{lemma} \label{deduction1}
{\em(i)} If $\Gamma\subseteq\Gamma'$ and $\Gamma\vdash_n\Sc$ then $\Gamma'\vdash_n\Sc$.

{\em(ii)} Let $\bar\Gamma=\{\Sc:\ \Gamma\vdash\Sc\}$. Then, $\bar\Gamma\vdash\Sc$ implies $\Gamma\vdash\Sc$.

{\em(iii)} If $\Gamma\vdash\Sc$, then $\Gamma_0\vdash\Sc$ for some finite $\Gamma_0\subseteq\Gamma$.

\end{lemma}
\begin{proof}
(i): Proceed by induction on $n$.

(ii): It is proved by induction on $n$ that for every $\Gamma$ and $\Sc$, \ \ $\bar\Gamma\vdash_n\Sc$ implies $\Gamma\vdash\Sc$.

(iii) Proceed by induction on $\vdash_n$.
\end{proof}

\begin{lemma}\label{deduction lemma}
For each $n$, if\ \ $\Gamma, 0\leqslant\theta \vdash_n \phi\leqslant\psi$, then
$\Gamma\vdash r\theta+\phi\leqslant\psi$ for some $r\geqslant0$.
\end{lemma}
\begin{proof}
We proceed by induction on $n$.
In case $n=0$, set $r=0$ if $\Vdash\phi\leqslant\psi$ or $\Sc\in\Gamma$ and set $r=1$ if the condition
$\phi\leqslant\psi$ coincides with $0\leqslant\theta$.
Suppose the claim is proved for every $k<n$ and consider the case (R2). So, assume for some $k<n$ one has that
$$\Gamma,0\leqslant\theta \ \vdash_k\ 0\leqslant s,\hspace{10mm}
\Gamma,0\leqslant\theta \ \vdash_k\ \phi\leqslant\psi.$$ We have two case:

$\bullet$ $s\geqslant0$: By the induction hypothesis, take $r\geqslant0$ such that
$$\Gamma\vdash\ r\theta+\phi\leqslant\psi$$ and multiply by $s$ to obtain the required result.

$\bullet$ $s<0$: By the induction hypothesis, take $r\geqslant0$ such that $\Gamma\vdash r\theta\leqslant s$.
Then, for sufficiently big $n$ one has that $$\Gamma\vdash nr\theta+s\phi\leqslant s\psi.$$

The cases (R1) and (R3) are similar.
\end{proof}

Repeating the argument, one shows that if
$\Gamma\cup\{0\leqslant\theta_1,\cdots,0\leqslant\theta_k\}\vdash_n\phi\leqslant\psi$,
then there are $r_1,...,r_k\geqslant0$ such that $\Gamma\vdash\ 0\leqslant \sum r_i\theta_i+\phi\leqslant\psi$.

\begin{lemma}\label{consequence2}
(i) If $\Gamma$ is consistent, then either $\Gamma,0\leqslant\phi$ is consistent or
there is $\epsilon>0$ such that $\Gamma\vdash\phi\leqslant-\epsilon$
(in which case $\Gamma,\phi\leqslant-\epsilon$ is consistent).

(ii) If $\Gamma\nvdash\phi\leqslant0$ then $\Gamma, 0\leqslant\phi$ is consistent.

(iii) If $\Gamma,-\epsilon\leqslant\phi$ is consistent for each $\epsilon>0$, then $\Gamma,0\leqslant\phi$ is consistent.
\end{lemma}
\begin{proof}
(i) Assume $\Gamma,0\leqslant\phi\vdash1\leqslant0$. Then, there is $r\geqslant0$ such that
$\Gamma\vdash r\phi+1\leqslant0$. Clearly, then $r\neq0$ and hence $\Gamma\vdash\phi\leqslant\frac{-1}{r}$.

(ii) is a consequences of (i).

(iii): Otherwise, by (i), $\Gamma\vdash\phi\leqslant-\epsilon$ for some $\epsilon>0$. This is a contradiction.
\end{proof}

\begin{lemma}
Assume $\Gamma,0\leqslant\theta\ \vdash\phi\leqslant\psi$ and $\Gamma,\theta\leqslant0\ \vdash\phi\leqslant\psi$.
Then, $\Gamma\vdash\phi\leqslant\psi$.
\end{lemma}
\begin{proof}
There are $r,s\geqslant0$ such that
$$\Gamma\ \vdash\ r\theta+\phi\leqslant\psi,\ \ \ \ \ \ \ \ \ \Gamma \vdash\ -s\theta+\phi\leqslant\psi.$$
If $r=0$ or $s=0$, we are done. Otherwise, we have that
$$\Gamma\ \vdash\ (r+s)\phi\leqslant(r+s)\psi$$ which implies the intended claim.
\end{proof}

\begin{lemma}\label{mean-proof0}
Assume both $\Gamma,0\leqslant\phi$ and $\Gamma,\phi\leqslant0$ are consistent. Then, $\Gamma,\phi=0$ is consistent.
\end{lemma}
\begin{proof}
Suppose that $$\Gamma,\phi\leqslant0,\ 0\leqslant\phi\ \ \vdash\ 1\leqslant0.$$
Then, there are $r,s\geqslant0$ such that $$\Gamma \ \vdash \ r\phi-s\phi+1\leqslant0.$$
Clearly, all cases $r<s$, $r=s$ and $r>s$ lead to contradiction.
\end{proof}

\section{Canonical model and completeness}\label{complete}

Let $\Gamma$ be a consistent theory.
By the axiom of choice and Lemma \ref{deduction1}, $\Gamma$ is contained in a maximal consistent theory.
Every maximal consistent theory $\Gamma$ is closed, i.e. $\Gamma\vdash\Sc$ implies that $\Sc\in\Gamma$.
Moreover, for every sentences $\phi,\psi$ either $\phi\leqslant\psi\in\Gamma$ or $\psi\leqslant\phi\in\Gamma$.

\begin{lemma}\label{equality} Let $\Gamma$ be a maximal consistent theory and $\phi$ be a proposition.
Then there is a unique $r$ such that $\phi=r\in \Gamma$.
\end{lemma}
\begin{proof} Let
$$r=\sup{\{u:\ u\leqslant\phi\in \Gamma}\},\hspace{15mm} s=\inf{\{u:\ \phi\leqslant u\in \Gamma}\}.$$
For every $\epsilon>0$, both $r-\epsilon\leqslant\phi$ and $\phi\leqslant s+\epsilon$
belong to $\Gamma$. Therefore, $r\leqslant s$.
Also, for each $\epsilon>0$ we have that $(r+\epsilon)\leqslant\phi\notin\Gamma$.
So, $\phi\leqslant r+\epsilon\in\Gamma$ and hence $s\leqslant r+\epsilon$. We conclude that $r=s$.
Note that $r\leqslant\phi\in\Gamma$ since $\Gamma,r\leqslant\phi$ is consistent.
Similarly, $\phi\leqslant r\in \Gamma$.
\end{proof}

Let $X$ be the set of all maximal consistent theories. For each $x\in X$ and proposition
$\phi$, the unique $r$ such that $\phi=r$ belongs to $x$ is denoted by $\phi^x$.
Let $\tau$ be the coarsest topology on $X$ for which every map $x\mapsto(\Box\phi)^x$ is lower semi-continuous.
In other word, $\tau$ is the smallest topology for which every set of the form
$$B(\phi,r)=\{x\in X:\ r<(\Box\phi)^x\}$$ is open.
So, finite intersections of such sets form a basis for $\tau$.
Define a valuation function by $\v_p(x)=p^x$. Then, $M=(X,\v)$ is called the \emph{canonical model} of AML.

\begin{proposition} (Truth lemma)
For each $\phi$, \ $\phi^M(x)=\phi^x$.
\end{proposition}
\begin{proof}
We use induction. The claim trivially holds for atomic $p$. The connective cases $r\phi$ and $\phi+\psi$ are obvious.
Assume the claim is proved for $\phi$. We prove it for $\Box\phi$. So, we must show that
$(\Box\phi)^M(x)=(\Box\phi)^x$ or that $\bunderline{\phi^M}(x)=(\Box\phi)^x$.
Let $(\Box\phi)^x=r$ and $\epsilon>0$ be given. Then, $x\vdash\Box\phi=r$ and hence
$$x\in\{y:\ r-\epsilon<(\Box\phi)^y\}=B(\phi,r-\epsilon).$$
So, as $\phi^y=\phi^M(y)$ by the induction hypothesis and $\vdash\Box\phi\leqslant\phi$, we have that
$$r-\epsilon\leqslant\inf\{\phi^M(y):\ y\in B(\phi,r-\epsilon)\}.$$
This implies that $r-\epsilon\leqslant\bunderline{\phi^M}(x)$. Since $\epsilon>0$ is arbitrary, we conclude that
$$r=(\Box\phi)^x\leqslant\bunderline{\phi^M}(x).$$

Now assume $\bunderline{\phi^M}(x)=r$ and $\epsilon>0$ is given.
Then, there is a basic open set $U\ni x$ such that (by induction hypothesis)
$$y\in U\ \ \Longrightarrow\ \ r-\epsilon\leqslant\phi^M(y)=\phi^y.$$
There are $\psi_1,...,\psi_k$ such that $$U=B(\Box\psi_1,0)\cap\cdots\cap B(\Box\psi_k,0).$$
This implies that for every $\delta>0$ $$\{\delta\leqslant\Box\psi_1,...,\delta\leqslant\Box\psi_k\}\vdash r-\epsilon\leqslant\phi.$$
So, by Lemma \ref{deduction lemma}, there are $r_1,...,r_k\geqslant0$ such that
$$\vdash\ \sum r_i(\Box\psi_i-\delta)+r-\epsilon\leqslant\phi.$$
Hence, by Lemma \ref{easy consequence} (vi), $$\vdash\ \sum r_i(\Box\psi_i-\delta)+r-\epsilon\leqslant\Box\phi.$$
We conclude that, as $\delta>0$ is arbitrary, for every $y\in U$ one has that $r-\epsilon\leqslant(\Box\phi)^y$.
In particular, $r-\epsilon\leqslant(\Box\phi)^x$ and hence $\bunderline{\phi^M}(x)\leqslant(\Box\phi)^x$ as $\epsilon>0$ is arbitrary.
\end{proof}

\begin{theorem}
{\em(Completeness)} Every consistent theory $\Gamma$ is satisfiable.
\end{theorem}
\begin{proof}
Let $M=(X,\v)$ be the canonical model and $x\in X$ be a maximal consistent theory containing $\Gamma$.
Then, for each $\phi$,\ \ $\phi^M=\phi^x$. In particular, if $\phi\leqslant\psi\in\Gamma\subseteq x$,
then $\phi^M=\phi^x\leqslant\psi^x=\psi^M$.
We conclude that $M\vDash_x\Gamma$.
\end{proof}

\begin{theorem} {\em (Approximate completeness)}
If $\Gamma\vDash0\leqslant\phi$, then $\Gamma\vdash-\frac{1}{n}\leqslant\phi$ for all $n$.
\end{theorem}
\begin{proof}
Otherwise, there is a maximal consistent $x\in X$ containing $\Gamma,\phi\leqslant-\frac{1}{n}$ for some $n$.
Then, $M\vDash_x\Gamma$ and $-\frac{1}{n}\leqslant\phi^M(x)$.
On the other hand, by the assumption, $0\leqslant\phi^M(x)$. This is a contradiction.
\end{proof}

\begin{theorem}\label{Weak completeness} {\em (Weak completeness)}
$\vDash0\leqslant\phi$ if and only if $\Vdash-\frac{1}{n}\leqslant\phi$ for all $n\geqslant1$.
\end{theorem}
\begin{proof}
Suppose $\nVdash-\frac{1}{n}\leqslant\phi$ for some $n$. Then, $\phi\leqslant-\frac{1}{n}$
is contained in a maximal consistent theory $x$. So, $M\nvDash_x0\leqslant\phi$ and hence $\nvDash0\leqslant\phi$.
The inverse direction is a consequence of Theorem \ref{soundness}.
\end{proof}

\section{Affine compactness}\label{compact}

In this section we prove affine compactness theorem by the ultramean construction.
An \emph{ultracharge} on an index set $\Omega$ is a finitely additive probability measure $\mu$ defined on the power set of $\Omega$.
Let $(\Omega,\mu)$ be an ultracharge space and for each $i\in\Omega$, $M_i=(X_i,\v_i)$ be a model.
For $(a_i),(b_i)\in\prod_iX_i$ set $(a_i)\sim(b_i)$ if $\mu\{i: a_i=b_i\}=1$. The equivalence class of $(a_i)$ is denoted by $[a_i]$.
The set of all such classes is denoted by $X$.
Let $U_i\subseteq X_i$ be open and set $$[U_i]=\{[a_i]:\ \mu\{i:a_i\in U_i\}=1\}.$$
Then, $[U_i]\cap[V_i]=[U_i\cap V_i]$ and these sets form a basis of a topology $\tau$ on $X$.
Finally, for each atomic $p$ let (see \cite{Rao} for finitely additive integration)
$$\v_p([a_i])=\int\v_{ip}(a_i)d\mu.$$ Then, $M=(X,\v)$ is a topological model which we denote by $\prod_\mu M_i$.

\begin{lemma}
For every proposition $\phi$ and $x\in M$, one has that $\phi^M(x)=\int\phi^{M_i}(x_i)d\mu$.
\end{lemma}
\begin{proof}
The claim is proved by induction on the complexity of $\phi$.
The atomic cases hold by the definition. The cases $r\phi$ and $\phi+\psi$ are clearly true.
Assuming the claim is proved for $\phi$, we prove it for $\Box\phi$. So, we must show that
$$(\Box\phi)^M(x)=\int(\Box\phi)^{M_i}(x_i)d\mu\hspace{15mm} \forall x\in X.$$
In other words, for each fixed $a=[a_i]$, we have to prove that
$$\bunderline{\phi^M}(a)=\int\bunderline{\phi^{M_i}}(a_i)d\mu.$$
Let $\epsilon>0$ be given. Then, there is an open $U_i\ni a_i$ such that for all $x_i\in U_i$
$$\bunderline{\phi^{M_i}}(a_i)-\epsilon\leqslant\phi^{M_i}(x_i).$$
So, $a\in[U_i]$ and for all $x=[x_i]\in[U_i]$
$$\int\bunderline{\phi^{M_i}}(a_i)d\mu-\epsilon\leqslant\int\phi^{M_i}(x_i)d\mu=\phi^M(x).$$
Since $\epsilon>0$ is arbitrary, we conclude that $$\int\bunderline{\phi^{M_i}}(a_i)d\mu\leqslant\bunderline{\phi^M}(a).$$
Conversely, assume $\epsilon>0$ and $[U_i]$ is an arbitrary basic open set containing $a=[a_i]$.
We may assume $a_i\in U_i$ for all $i$. So, there is $x_i\in U_i$ such that
$$\phi^{M_i}(x_i)\leqslant\bunderline{\phi^{M_i}}(a_i)+\epsilon.$$
Then, by integrating, for $x=[x_i]\in[U_i]$ we have that $$\phi^{M}(x)\leqslant\int\bunderline{\phi^{M_i}}(a_i)d\mu+\epsilon.$$
This implies that $$\bunderline{\phi^M}(a)\leqslant\int\bunderline{\phi^{M_i}}(a_i)d\mu.$$
\end{proof}

The \emph{affine closure} of set $\Gamma$ of conditions is the family of all conditions of the form
$$r_1\phi_1+\cdots+r_k\phi_k\leqslant r_1\psi_1+\cdots+r_k\psi_k$$ where $\phi_i\leqslant\psi_i$ belongs to $\Gamma$
and $r_i\geqslant0$. We say $\Gamma$ is \emph{affinely satisfiable} if every condition in its affine closure is satisfiable.

\begin{theorem} (Affine compactness)
Every affinely satisfiable set $\Gamma$ of conditions is satisfiable.
\end{theorem}

\begin{proof}
Let $\Gamma$ be an affinely satisfiable theory which we may further assume it is maximal with this property.
For each $\phi$ set $$Q(\phi)=\inf\{r:\phi\leqslant r\in\Gamma\}.$$
By maximality, $Q$ is defined for every $\phi$ and it is sublinear, i.e.
$$Q(\phi+\psi)\leqslant Q(\phi)+Q(\psi),\hspace{16mm} Q(r\phi)=rQ(\phi)\ \ \ \ \mbox{for}\ r\geqslant0.$$
Moreover, $\phi\leqslant Q(\phi)$ belongs to $\Gamma$. Let $T_0$ be the identity map on the linear subspace $\Rn\subseteq V$.
Then, $T_0\leqslant Q$ on $\Rn$ and by the Hahn-Banach extension theorem (\cite{Aliprantis-Inf}, Th 8.30),
$T_0$ extends to a linear map $T$ on $V$ such that $T(\phi)\leqslant Q(\phi)$ for every $\phi$.
Note that $T$ is positive. In particular, if $\vDash\phi\leqslant0$ then $\phi\leqslant0\in\Gamma$
and hence $T(\phi)\leqslant Q(\phi)\leqslant0$.

Our favorite index set $\Omega$ is the family of all propositions (not to be confused with $V$).
Since, $\phi\leqslant Q(\phi)$ belongs to $\Gamma$, for each $i=\phi$ there is a model $M_i=(X_i,\v_i)$
and $a_i\in X_i$ such that $\phi^{M_i}(a_i)\leqslant Q(\phi)$.
Put the discrete topology on $\Omega$. Then, we may write $\Rn\subseteq V\subseteq\mathbf{C}_b(\Omega)$
if we identify $\phi\in V$ with the map $i\mapsto\phi^{M_i}$.
Since $V$ majorizes $\mathbf{C}_b(\Omega)$ (as every $f\in\mathbf{C}_b(\Omega)$ is bounded),
by the Kantorovich extension theorem \cite{Aliprantis-Inf}, $T$ is extended to a positive linear functional
$\bar{T}$ on $\mathbf{C}_b(\Omega)$.
By (a variant of) Riesz representation theorem \cite{Aliprantis-Inf} there is a (maximal) probability charge $\mu$ on $\Omega$
such that for every $\phi$ $$\bar{T}(\phi)=\int\phi^{M_i}d\mu.$$
Let $M=\prod_{\mu}M_i$ and $a=[a_i]$. Then, for each proposition $\phi$
$$\phi^M(a)=\int\phi^{M_i}(a_i)d\mu=T(\phi)\leqslant Q(\phi).$$
Clearly, if $\phi\leqslant\psi\in\Gamma$ then $Q(\phi-\psi)\leqslant0$ and hence $\phi^M(a)\leqslant\psi^M(a)$.
In particular, $M\vDash_a\Gamma$.
\end{proof}

An easy consequence of affine completeness and affine compactness theorems is that if both $\Gamma,0\leqslant\phi$ and $\Gamma,\phi\leqslant0$
are consistent, then so is $\Gamma,\phi=0$. The following proposition justifies that AML is in some sense continuous.

\begin{lemma} \label{proofs}
Let $S$ be a set of conditions of the form $0\leqslant\sigma$.

\emph{(i)} If $\Gamma\cup S\vDash0\leqslant\theta$, then for each $\epsilon>0$ there exists
$0\leqslant\sigma$ in the affine closure of $S$ such that $\Gamma\vDash\sigma\leqslant\theta+\epsilon$.
In particular, $\Gamma,0\leqslant\sigma\vDash0\leqslant\theta+\epsilon$.

\emph{(ii)} If $\Gamma,0\leqslant\sigma\vDash0\leqslant\theta$, for each $\epsilon>0$ there exists
$\delta>0$ such that $\Gamma,-\delta\leqslant\sigma\vDash-\epsilon\leqslant\theta$.
\end{lemma}
\begin{proof} (i): Otherwise, there exists $\epsilon>0$ such that
for every $0\leqslant\sigma$ in the affine closure of $S$, the theory
$\Gamma\cup\{\theta+\epsilon\leqslant\sigma\}$ is satisfiable.
This implies that $\Gamma\cup S\cup\{\theta+\epsilon\leqslant0\}$ is affinely satisfiable.\
(ii): By (i), there exists $r\geqslant0$
such that $\Gamma\vDash r\sigma\leqslant\theta+\frac{\epsilon}{2}$.
If $r=0$ we have that $\Gamma\vDash-\epsilon\leqslant\theta$.
Otherwise, $\Gamma,\frac{-\epsilon}{2r}\leqslant\sigma\vDash-\epsilon\leqslant\theta$.
\end{proof}

Let $V$ be the ordered vector space of propositions where we identify $\phi$, $\psi$ if
$\vdash\phi=\psi$ and we define $\phi\leqslant\psi$ if $\vdash\phi\leqslant\psi$.
Also, a norm is defined on $V$ by setting $$\|\phi\|=\sup\{|\phi^M(x)|:\ M=(X,\v)\ \textrm{is a model and}\ x\in X\}.$$
Thus, $V$ is a partially ordered normed vector space.
Any maximal consistent theory $x\in X$ determines a positive linear map $\phi\mapsto\phi^x$ from $V$ to $\Rn$, i.e.
$x(1)=1$ and $$(r\phi+\psi)^x=r\phi^x+\psi^x$$
$$0\leqslant\phi \ \ \Longrightarrow \ \ 0\leqslant\phi^x.$$

\begin{lemma}
Let $T:V\rightarrow\Rn$ be a positive linear functional such that $T(1)=1$.
Then, the set of conditions of the form $\phi=T(\phi)$ is satisfiable.
\end{lemma}
\begin{proof}
We claim that every condition $\phi\leqslant T(\phi)$ is satisfiable. Suppose not.
Then, for some $\epsilon>0$ one has that $\vDash T(\phi)+\epsilon\leqslant\phi$.
By Theorem \ref{Weak completeness} and positivity of $T$, we have that $T(\phi)+\epsilon\leqslant T(\phi)$.
This is a contradiction. Similarly, $T(\phi)\leqslant\phi$ and hence $\phi=T(\phi)$ is satisfiable.
\end{proof}

We conclude that every $x\in X$ corresponds to a positive linear map $T:V\rightarrow\Rn$.
Note that if $T_1,T_2$ are positive linear, then so is $\lambda T_1+(1-\lambda)T_2$ for every $\lambda\in[0,1]$.
In fact, the family of positive linear maps on $V$ is a compact convex subset of $V^*$ (the continuous dual of $V$).
The topology of $X$ is however weaker than the weak-star topology of $V^*$.
It follows that the canonical model $M=(X,\v)$ is compact convex (but not necessarily Hausdorff).

\end{document}